\documentclass[10pt]{amsart}
\usepackage{amsmath,amssymb,amsthm}
\usepackage{graphicx}
\usepackage{color}
\newtheorem{theorem}{Theorem}    

\newtheorem{corollary}{Corollary}
\theoremstyle{definition}

\newtheorem*{remark*}{Remark}
\newcommand{\Z}{\mathbb{Z}}

\title[Orientation-reversing distance one surgeries]{A note on orientation-reversing distance one surgeries on non-null-homologous knots}
\author{Tetsuya Ito}
\begin{document}

\begin{abstract}
We show that there are no distance one surgeries on non-null-homologous knots in $M$ that yield $-M$ ($M$ with opposite orientation) if $M$ is a 3-manifold obtained by a Dehn surgery on a knot $K$ in $S^{3}$, such that the order of its first homology is divisible by $9$ but is not divisible by $27$. 

As an application, we show several knots, including the $(2,9)$ torus knot, do not have chirally cosmetic bandings. This simplifies a proof of the result first proven by Yang; the $(2,k)$ torus knot $(k>1)$ has a chirally cosmetic banding if and only if $k=5$.
\end{abstract} 

\maketitle


The Dehn surgery on a knot $K$ in a 3-manifold $M$ along the slope $\alpha$ is called \emph{distance one surgery} if $\Delta(\alpha,\mu_K)=1$. Here $\mu_K$ is the meridian of $K$ and $\Delta$ denotes the distance of two slopes, i.e., the minimum geometric intersection number. When $M=S^{3}$, distance one Dehn 
surgeries are usually called the integral surgeries.

For a $\Z/2$-homology sphere $Y$, let $\mathfrak{t}_Y$ be the self-conjugate Spin${}^{c}$ structure of $Y$ and $d(Y,\mathfrak{t}_Y)$ be the $d$-invariant (correction term) of $Y$ \cite{os}.

The aim of this note is to give the following non-existence result on distance one surgeries that reverses the orientation.

\begin{theorem}
\label{theorem:main}
Let $Y$ be an L-space obtained by Dehn surgery on a knot in $S^{3}$, and that $|H_1(Y;\Z)|= 9p_0$ with $p_0\not \equiv 0 \pmod{3}$. If $d(Y,\mathfrak{t}_Y) \neq 0,\pm 1$, then there are no distance one surgeries between $Y$ and $-Y$. 
\end{theorem}
Here $-Y$ denotes the 3-manifold $Y$ with opposite orientation.

Our proof of Theorem \ref{theorem:main} has the same spirit as \cite{it-casson}. We show the non-existence of distance one surgery along non-null-homologous knot by using the Casson-Walker invariant.

\begin{theorem}
\label{theorem:main2}
Let $M$ be a 3-manifold which is obtained by a $\frac{p}{q}$-surgery on a knot $K$ in $S^{3}$ ($p,q>0$). If $p=9p_0$ with $p_0\not \equiv 0 \pmod{3}$, then there are no distance one surgeries on a non-null-homologous knot in $M$ that produce $-M$.
\end{theorem}

Using the results of \cite{mv} that treats null-homologous knots in L-space, Theorem \ref{theorem:main} follows from Theorem \ref{theorem:main2}.

\begin{proof}[Proof of Theorem \ref{theorem:main}]
By Theorem \ref{theorem:main2}, a knot $K$ in $Y$ whose distance one surgery yields $Y$ is null-homologous. However, Moore and Vazquez showed that if there is a distance one surgery between L-space $Y$ and $-Y$ along a null-homologous knot, then $|d(Y,\mathfrak{t}_Y)-d(-Y,\mathfrak{t}_{-Y})| \in \{0,2\}$ \cite[Theorem 3.5]{mv}. Since $d(-Y,\mathfrak{t}_{-Y}) = - d(Y,\mathfrak{t}_Y)$, it implies that $d(Y,\mathfrak{t}_Y)\in \{ 0,\pm 1 \}$. This is a contradiction.
\end{proof}

Recently the distance one surgeries between lens spaces, especially the lens space $L(n,1)$, has been actively studied \cite{lmv,mv,wy,ya}. They are regarded as a generalization of the lens space realization problem solved in \cite{gr}, the problem to determine which lens space can appear as a Dehn surgery on a non-trivial knot in $S^{3}$ and are also motivated from DNA topologies. The main tool of these studies are the Heegaard Floer homology, especially the $d$-invariant and its surgery formulas.

The distance one surgery has an application to band surgeries. 
Let $L$ be an unoriented link in $S^{3}$ and $b:I \times I \rightarrow S^{3}$ be an embedding such that $b(I\times I) \cap L = b(I \times \partial I)$ where $I=[0,1]$.
Then we get a new link $L'=(L - b(I \times \partial I)) \cup b(\partial I \times I)$. We say that $L'$ is obtained from $L$ by a \emph{band surgery} along the band $b$.  A \emph{chirally cosmetic banding} of a knot $K$ is a band surgery that produces its mirror image $\overline{K}$.

By the Montesinos trick, if a knot $K'$ is obtained from a knot $K$ by a band surgery, then there is a distance one surgery between their double branched coverings $\Sigma_2(K')$ and $\Sigma_2(K)$. Thus one can use the non-existence of distance one surgeries to obstruct band surgeries. If an L-space $Y$ is the double branched covering of a quasi-alternating knot $K$, then $d(Y,\mathfrak{t}_Y)=4\sigma(K)$, where $\sigma(K)$ is the signature of $K$ \cite{lo}. Thus Theorem \ref{theorem:main} implies the following.

\begin{corollary}
\label{cor:cosmetic}
Let $K$ be a quasi-alternating knot whose double branched covering $\Sigma_2(K)$ is obtained by a Dehn surgery on a knot in $S^{3}$. If $\det(K) = 9d$ with $d \not \equiv 0 \pmod{3}$ and $\sigma(K)\neq 0, \pm 4$, then $K$ has no chirally cosmetic banding.
\end{corollary}

In \cite{li}, by using Casson-Gordon invariants Livingston showed that $(2,k)$ torus knot $T(2,k)$ $(k>1)$ has a chirally cosmetic banding, then $k=5$ or $9$. The chirally cosmetic banding of $T(2,5)$ is known \cite{ze} but it remained unsolved whether $T(2,9)$ has chirally cosmetic banding or not. Corollary \ref{cor:cosmetic} shows that $T(2,9)$ has no chirally cosmetic banding hence we conclude the following, which was first proven in \cite{ya}.

\begin{corollary}\cite[Corollary 1.5]{ya}
\label{cor:torus}
The torus knot $T(2,k)$ $(k>1)$ admits chirally cosmetic banding if and only if $k=5$.
\end{corollary}

Our proof still uses the Heegaard Floer homology, results from \cite{mv} based on the surgery formula of the $d$-invariant for a null-homologous knot in L-space \cite{nw}. Nevertheless, our argument allows us to avoid technically hard Heegaard Floer homology arguments developed in \cite{ya}. It is interesting and surprising that the Casson-Walker invariant, one of the most fundamental invariant of 3-manifolds, can be used to solve such a remaining subtle case.

\begin{proof}[Proof of Theorem \ref{theorem:main2}]

Assume to the contrary that there is a non-null-homologous knot $K_0$ in $M$ whose distance one surgery is $-M$.

Since $M$ is the $\frac{p}{q}$-surgery on the knot $K=K_y$, we view $M$ as $M=(S^{3}-N(K_y)) \cup (S^{1}\times D^{2})$ where $N(K_y)$ is the open tubular neighborhood of $K_y$. We put the knot $K_0 \subset M$ so that 
\[ K_0 \subset S^{3}-N(K_y) \subset (S^{3}-N(K_y)) \cup (S^{1}\times D^{2}) =M \]

Let $K_x$ be the knot in $S^{3}$ obtained as the image of $K_0$ under the natural inclusion map $\iota: S^{3}-N(U) \hookrightarrow S^{3}$. Let $\alpha$ be the surgery slope of $K_0$. Since $\iota$ sends the meridian of $K_0$ to the meridian of $K_x$, the image of the slope $\alpha$ is also distance one, namely, the integral slope $m$ for some $m \in \Z$.

Let $L=K_x \cup K_y$ be rationally framed two component link, where $K_x$ has the integral framing $m$ and $K_y$ has the rational framing $\frac{p}{q}$. Let $S^{3}_L$ be the 3-manifold obtained by Dehn surgery along $L$. By the construction, $S^{3}_L$ is equal to $-M$.

Let $\ell$ be the linking number of $L$. Since $|H_1(S^{3}_L;\Z)| =|H_1(M;\Z)|= p$, it follows that 
\begin{equation}
\label{eqn:1sthomology}
mp-q\ell^{2} = \varepsilon p \quad (\varepsilon \in \{\pm 1\}) 
\end{equation} 
Since $p$ and $q$ are coprime this shows that $p \mid \ell^{2}$. On the other hand, since $K_0$ is not null-homologous in $M$, $p \nmid \ell$. 
Since $p=9p_0$ with $p_0 \not \equiv 0 \pmod{3}$, we may write
\begin{equation}
\label{eqn:non-null}
\ell^{2} = p \ell_0 , \quad \ell_0 \not \equiv 0 \pmod{3}
\end{equation} 

Let $\lambda$ be the Casson-Walker invariant of rational homology spheres.
By the Casson-Walker's knot surgery formula \cite{wa}
\[ \lambda(S^{3}_L) = \lambda(-M)=-\lambda(M) = -\frac{q}{p}a_2(K_y)+\frac{1}{2}s(q,p)\]
where $s(q,p)$ denotes the Dedekind sum.
On the other hand, by applying the two component link rational surgery formula \cite[Theorem 1.5]{it-casson} we get
\begin{align*}
\frac{mp-q\ell^{2}}{q} \left( \frac{\lambda(S^{3}_L)}{2}-\frac{\sigma}{8} \right)
& =  \frac{p}{q}a_2(K_x)  -\frac{p}{12q} +\frac{p\ell^2}{24q} + m a_2(K_y) -\frac{m}{24} -\frac{m}{24q^{2}} + \frac{m\ell^2}{24}\\
& \quad + 2v_3(L) + \frac{mp-q\ell^2}{24q}\left(12s(m,1) - m +12s(p,q)-\frac{p}{q} \right)
\end{align*}
Here
\begin{itemize}
\item $\sigma$ is the signature of the linking matrix $\begin{pmatrix}m & \ell \\ \ell & \frac{p}{q} \end{pmatrix} \in \Z$. 
\item $a_2(K_x), a_2(K_y) \in \Z$ is the second coefficient of the Conway polynomial of $K_x$ and $K_y$.
\item $2v_3(L)= -a_3(L)+(a_2(K_x)+a_2(K_y))\ell + \frac{1}{12}(\ell^{3}-\ell) \in \frac{1}{2}\Z$ is an invariant of the link $L$, where $a_3(L)$ is the third coefficient of the the Conway polynomial $\nabla_L(z)$ of the link $L$.
\end{itemize}
Actually these precise definitions are not important in the following argument. We will just use the facts that $\sigma, a_2(K_x), a_2(K_y), 4v_3(L) \in \Z$. 

By the reciprocity low of the Dedekind sums
\[ s(p,q)+s(q,p) = -\frac{1}{4}+ \frac{1}{12}\left( \frac{p}{q} + \frac{q}{p} + \frac{1}{pq}\right) \]
it follows that
\[ 12s(m,1) - m +12s(p,q)-\frac{p}{q} = -m -12s(q,p)+\frac{q}{p}+\frac{1}{pq}-3 \]
Thus by \eqref{eqn:1sthomology} we get
\begin{align*}
&24\varepsilon p \left(-a_2(K_y)\frac{q}{2p}+ \frac{1}{4}s(q,p)-\frac{\sigma}{8} \right) \\
& \quad = 24p a_2(K_x) -2p  + p\ell^2 + 24qma_2(K_y) -m -\frac{m}{q} + mq\ell^{2} + 12q(4v_3(L))\\
& \qquad \quad + \varepsilon p\left(-m -12s(q,p)+\frac{q}{p}+\frac{1}{pq}-3 \right)
\end{align*}
Hence 
\begin{align*}
&-12\varepsilon q a_2(K_y) + 3\varepsilon \left( 6ps(q,p) \right)- 3\varepsilon p\sigma \\
&\  = 24p a_2(K_x) +24qma_2(K_y) -2p  + (\varepsilon-m)\left(q+ \frac{1}{q}\right) + \ell^{2}(mq+p) + 12q(4v_3(L)) - \varepsilon p(1+m) \\
&\ =  24p a_2(K_x)+24qma_2(K_y) -2p  -\ell_0(q^2+1) + \ell^{2}(mq+p) + 12q(4v_3(L)) - \varepsilon p(1+m) \\
\end{align*}
Here at the last equality we use
\[ \varepsilon -m =-\frac{q\ell^{2}}{p}=-q\ell_0 \]
that follows from \eqref{eqn:1sthomology} and \eqref{eqn:non-null}.
Since $6p s(q,p) \in \Z$, by taking equality modulo $3$, we get 
\[ 0 \equiv -\ell_0(q^2+1) \pmod{3}\]
Since $\ell_0, (q^2+1) \not \equiv 0 \pmod{3}$ this is a contradiction.
\end{proof}

\section*{Acknowledgement}
The author is partially supported by JSPS KAKENHI Grant Numbers 19K03490, 21H04428, 	23K03110. The author wishes to express his gratitude to Kazuhiro Ichihara for discussions and comments.


\begin{thebibliography}{1}



\bibitem[Gr]{gr}
J.\ Greene, 
{\em The lens space realization problem.}
Ann. of Math. (2)177(2013), no.2, 449--511.
\bibitem[It]{it-casson}
T.\ Ito,
{\em Applications of the Casson-Walker invariant to the knot complement and the cosmetic crossing conjectures.}
Geom. Dedicata 216(2022), no.6, Paper No. 63, 15 pp.

\bibitem[LMV]{lmv}
T.\ Lidman, A.\ Moore, and M.\ Vazquez, 
{\em Distance one lens space fillings and band surgery on the trefoil knot.}
Algebr. Geom. Topol.19(2019), no.5, 2439--2484.


\bibitem[LO]{lo}
P.\ Lisca, and B.\ Owens,
{\em Signatures, Heegaard Floer correction terms and quasi-alternating links.}
Proc. Amer. Math. Soc.143(2015), no.2, 907--914.

\bibitem[Li]{li}
C.\ Livingston,
{\em Chiral smoothings of knots.}
Proc. Edinb. Math. Soc. (2)63(2020), no.4, 1048--1061.

\bibitem[MV]{mv}
A.\ Moore and M.\ Vazquez,
{\em A note on band surgery and the signature of a knot.}
Bull. Lond. Math. Soc.52(2020), no.6, 1191--1208.

\bibitem[NW]{nw}
Y.\ Ni, and Z.\ Wu,
{\em Cosmetic surgeries on knots in $S^3$.}
J. Reine Angew. Math.706(2015), 1--17.

\bibitem[OS]{os}
P.\ Ozsv\'ath, and Z.\ Szab\'o, 
{\em Absolutely graded Floer homologies and intersection forms for four-manifolds with boundary.}
Adv. Math.173(2003), no.2, 179--261.

\bibitem[Wa]{wa}
K.\ Walker,
{\em An extension of Casson's invariant.}
Ann. of Math. Stud., 126
Princeton University Press, Princeton, NJ, 1992. vi+131 pp.

\bibitem[WY]{wy}
Z.\ Wu, and J.\ Yang, 
{\em Studies of distance one surgeries on the lens space $L(p,1)$.}
Math. Proc. Cambridge Philos. Soc.172(2022), no.2, 267--301.

\bibitem[Ya]{ya}
J.\ Yang, 
{\em Distance one surgeries on the lens space $L(n,1)$.}
arXiv:2108.06199v2

\bibitem[Ze]{ze} A. Zekovi\'c,
{\em Computation of Gordian distances and $H_2$-Gordian distances of knots,}
Yugosl. J. Oper. Res. \textbf{25} (2015), 133--152.


\end{thebibliography}
\end{document}